\documentclass[12pt,reqno]{amsart}%
\usepackage{amsmath}
\usepackage[margin=1.5in]{geometry}
\usepackage{graphicx}
\usepackage{amsfonts}
\usepackage{amssymb}%
\usepackage[utf8]{inputenc}
\usepackage{comment}
\usepackage{hyperref}
\usepackage{mathtools}

\newtheorem{theorem}{Theorem}[section]
\theoremstyle{plain}

\newtheorem{claim}{Claim}

\newtheorem{lemma}{Lemma}[section]

\newtheorem{proposition}{Proposition}[section]

\numberwithin{equation}{section}
\theoremstyle{definition}
\newtheorem{definition}{Definition}
\theoremstyle{remark}
\newtheorem{remark}{Remark}[section]

\def\e{\varepsilon}
\def\al{\alpha}
\def\pd{\partial}
\def\re{\mathbb{R}}
\def\la{\lambda}
\def\La{\Lambda}
\def\mc{\mathcal}

\def\mbb{\mathbb}
\DeclarePairedDelimiter{\floor}{\lfloor}{\rfloor}
\newcommand{\eqal}[1]{\begin{equation}\begin{aligned}#1\end{aligned}\end{equation}}
\newcommand{\es}[1]{\begin{align*}#1\end{align*}}

\title[Lagrangian mean curvature type equations]{
Optimal regularity for Lagrangian mean curvature type equations
}
\author{Arunima Bhattacharya}
\address{Department of Mathematics, Phillips Hall\\
 the University of North Carolina at Chapel Hill, NC }
\email{arunimab@unc.edu}

\author{Ravi Shankar}
\address{Department of Mathematics, Fine Hall\\
Princeton University, Princeton, NJ}
\email{rs1838@princeton.edu}

\begin{document}


\maketitle

\begin{abstract}
We classify regularity for Lagrangian mean curvature type equations, which include the potential equation for prescribed Lagrangian mean curvature and those for Lagrangian mean curvature flow self-shrinkers and expanders, translating solitons, and rotating solitons. Convex solutions of the second boundary value problem for certain such equations were constructed by Brendle-Warren 2010, Huang 2015, and Wang-Huang-Bao 2023.  We first show that convex viscosity solutions are regular provided the Lagrangian angle or phase is $C^2$ and convex in the gradient variable.  We next show that for merely H\"older continuous phases, convex solutions are regular if they are $C^{1,\beta}$ for sufficiently large $\beta$. Singular solutions are given to show that each condition is optimal and that the H\"older exponent is sharp. Along the way, we generalize the constant rank theorem of Bian and Guan to include arbitrary dependence on the Legendre transform.
\end{abstract}

\section{Introduction}

In this paper, we classify regularity for convex viscosity solutions of Lagrangian mean curvature type equations
\eqal{
\label{slag}
\sum_{i=1}^n\arctan\lambda_i=\psi(x,u,Du)\in[0,n\pi/2]
}
where $\lambda_i$'s are the eigenvalues of the Hessian of $u$. The left-hand side $\theta(x):=tr\arctan D^2u(x)$ is called the Lagrangian phase or angle of the Lagrangian submanifold $(x,Du(x))\subset \mbb C^n=\re^n\times\re^n$. 

\smallskip
The geometric motivation for studying Lagrangian mean curvature type equations is the presence of several interesting special cases of this equation, which we discuss below; see \cite{Y20} by Yuan for a recent survey.  From a PDE point of view, the existence of smooth convex solutions to the second boundary problem was shown for several such equations by Brendle-Warren \cite{BrW} and Wang-Huang-Bao \cite{WHB}, in the case of uniformly convex smooth domains.  By taking limits, one recovers viscosity solutions of convex domains satisfying the second boundary problem in a weak sense.

\smallskip
\textbf{Lagrangian mean curvature equation.} Suppose $\psi=\psi(x,y)$ is defined on the ambient space $(x,y)\in \mbb C^n$, and $u$ solves
\eqal{
\label{LMCp}
\sum_{i=1}^n\arctan\lambda_i=\psi(x,Du).
}
By Harvey-Lawson \cite[Proposition 2.17]{HL}, the Lagrangian angle determines the mean curvature vector of $(x,Du(x))$: 
\eqal{
\label{H}
H=J\nabla_g\theta,
}
where $g=I+(D^2u)^T(D^2u)$ is the induced metric on $(x,Du(x))$, and $J$ is the almost complex structure on $\mbb C^n$. Since $\theta(x)=\psi(x,Du(x))$, it follows that $J\nabla_g\psi(x,Du(x))=(J\bar\nabla\psi(x,y))^\bot$, where $\bot$ is the normal projection, and 
$\bar\nabla=(\pd_x,\pd_y)$ is the ambient gradient. Thus \eqref{LMCp} is the potential equation for prescribed Lagrangian mean curvature:
\eqal{
\label{LMC}
H=\Delta_g(x,Du(x))\stackrel{\eqref{LMCp}}{=}(J\bar\nabla\psi(x,y))^\bot|_{y=Du(x)},
}
The constant case
\eqal{
\label{s1}
\sum_{i=1}^n\arctan\lambda_i=c
}
is the \textit{special Lagrangian equation} of Harvey-Lawson \cite{HL}. In this case, $H=0$, and $(x,Du(x))$ is a volume-minimizing Lagrangian submanifold.  

\smallskip
For certain $\psi=\psi(x)$, smooth convex solutions were constructed by Wang-Huang-Bao \cite{WHB} satisfying $Du(\Omega_1)=\Omega_2$ for prescribed uniformly convex smooth domains $\Omega_i$, following Brendle-Warren \cite{BrW} for the constant $\psi$ case; see also Huang \cite{H15} for a construction using Lagrangian mean curvature flow.


\smallskip
\textbf{Lagrangian mean curvature flow.} A family of Lagrangian submanifolds $X(x,t):\re^n\times\re\to\mbb C^n$ evolves by Lagrangian mean curvature flow if it solves
\eqal{
\label{LMCF}
(X_t)^\bot=\Delta_gX.
}
After a change of coordinates, we can locally write $X(x,t)=(x,Du(x,t))$, such that $\Delta_gX=(J\bar\nabla\theta(x,t))^\bot$ by \eqref{H}.  This means a local potential $u(x,t)$ evolves by the parabolic equation
\eqal{
\label{ut}
u_t&=\sum_{i=1}^n\arctan\lambda_i,\\
&u(x,0):=u(x).
}
Several special cases of \eqref{slag} correspond to symmetry reductions of \eqref{LMCF}, which reduce \eqref{ut} to an elliptic equation for $u(x)$, as demonstrated, for example, by Chau-Chen-He \cite{CCH}.  Such solutions model singularities of the mean curvature flow.

\smallskip
\begin{enumerate}
\item[1.] \textit{Self-similar solutions.} 
If $u(x)$ solves
\eqal{
\label{self}
\sum_{i=1}^n\arctan\lambda_i=c+b(x\cdot Du-2u)
}
then $X(x,t)=\sqrt{1-2bt}\,(x,Du(x))$ is a shrinker or expander solution of \eqref{LMCF}, if $b>0$ or $b<0$, respectively.  The initial submanifold $(x,Du(x))$ has mean curvature $H=-bX^\bot$.  Entire smooth solutions to \eqref{self} for $b>0$ are quadratic polynomials, by Chau-Chen-Yuan \cite{CCY}; see also Huang-Wang \cite{HW} for the smooth convex case.  The circle $x^2+u'(x)^2=1$ is a closed example of a shrinker $b=1,c=0$ in dimension $n=1$.  See e.g. Joyce-Lee-Tsui \cite{JLT} for other non-graphical examples.  

\smallskip
\item[2.] \textit{Translating solitons.} 
If $u(x)$ solves
\eqal{
\label{tran}
\sum_{i=1}^n\arctan\lambda_i=c+k\cdot x+\ell\cdot Du(x),
}
then $X(x,t)=(x,Du(x))+t(-\ell,k)$ is a \textit{translator} solution of \eqref{LMCF}, with ``constant" mean curvature $H=(-\ell,k)^\bot$.  The ``grim reaper" curve $(x,u'(x))=(x,-\ln\cos(x))$, for $n=1$ and $k=1,\ell=c=0$, is the model example.  When $\ell=0$, \eqref{tran} corresponds to the case of prescribed constant mean curvature vector of the Lagrangian graphs, see \cite{WHB}. Entire solutions to \eqref{tran} with Hessian bounds are quadratic polynomials, by Chau-Chen-He \cite{CCH}; see also Ngyuen-Yuan \cite{NY} for entire ancient solutions to \eqref{ut} with Hessian conditions.

\medskip
\item[3.] \textit{Rotating solitons.} If $A\in SU(n)$ is skew-adjoint, then the Hamiltonian vector field $A\cdot z=J\bar \nabla\psi$ has a real potential $\psi(x,y)=\frac{1}{2i}\langle z,A\cdot z\rangle_{\mbb C^n}$.  Since $\exp(tA)\in U(n)$ preserves the symplectic form 
$dz\wedge d\bar z=\sum dz^i\wedge d\bar z^i$, the Hamiltonian flow $X(x,t)=\exp(tA)(x,Du(x))$ is a Lagrangian immersion with $X_t=AX=J\bar \nabla\psi$.  Provided $\psi(x,Du(x))=\theta(x)$, i.e. $u(x)$ solves \eqref{LMCp}, it follows from this and the isometric property of $\exp(tA)$ that $X(x,t)$ evolves by mean curvature flow \eqref{LMCF}.  For the example $A=aJ$ and $\psi(x,y)=c+\frac{a}{2}|z|^2$, if $u(x)$ solves
\eqal{
\label{rotator}
\sum_{i=1}^n\arctan\lambda_i=c+\frac{a}{2}(|x|^2+|Du|^2),
}
then $X(x,t)=\exp(atJ)(x,Du(x))$ is a \textit{rotator} solution of \eqref{LMCF}, with $H=a(JX)^\bot$.  The ``Yin-Yang" curve of Altschuler \cite{Alt} is the solution in dimension $n=1$; see also the notes of Yuan \cite[pg. 3]{YNotes}.  The classical radially symmetric solutions of \eqref{rotator} are studied in \cite{LB} for $a=0$.  Other examples of such equations include $\psi(x,y)=-\langle x,B\cdot y\rangle$ for $B\in SO(n)$ skew-symmetric and $\psi(x,y)=Re(\bar z_1z_2)=x_1x_2+y_1y_2$ in $\mbb C^2$, with $A=diag(B,B)$ and $A\cdot(x,y)=(-y_2,-y_1,x_2,x_1)$, respectively.  
\end{enumerate}

\medskip
Optimal regularity was previously obtained for the ``dual" form of \eqref{s1} under the strict convexity condition. The Monge-Amp\`ere equation 
\eqal{
\label{MA}
\ln\det D^2u=\sum_{i=1}^n\ln\la_i=c,
}
is the potential equation for special Lagrangian submanifolds in $(\mathbb {C}^n, dx\,dy)$, as introduced by Hitchin \cite{Hi}; see also Mealy \cite{Me} for an equivalent form.  
Formally, this metric can be brought to the special Lagrangian one $dx^2+dy^2$ after a suitable complex rotation.  
In dimension $n=2$, Alexandrov \cite{AL} and Heinz \cite{H} established strict convexity of solutions to Monge-Amp\`ere type equations, then regularity for generalized solutions.  However, in higher dimensions, the $C^{1,1-2/n}$ and $W^{2,p}$ for $p<\frac{n(n-1)}{2}$ singular solution of Pogorelov \cite{P2} shows that the strict convexity condition is 
optimal for regularity, since its graph contains a line.  Pogorelov's interior estimate \cite{P2} indicated that strict convexity is sufficient for regularity, and Urbas \cite{U} showed that solutions in $C^{1,\al}$ or $W^{2,p}$ for $\al>1-\frac{2}{n}$ or $p>\frac{n(n-1)}{2}$ are strictly convex, then interior regular.  The $W^{2,n(n-1)/2}$ case is due to Collins-Mooney \cite{CM}.  By Caffarelli \cite{Caf}, regularity holds even for positive H\"older right-hand side $\det D^2u\in C^\al$, provided $u=0$ on the (convex) boundary.

Before presenting the main results, we first introduce the following notation and definition. \\

 \noindent
 \textbf{Notations.} \begin{itemize}
\item[1.] 
For $a>0$, we denote by $C^{a}$ the H\"older space $C^{k,\alpha}$, where $k=\floor{a}$ is the integer part, and $\alpha=a-k$ is the fractional part.  If $\alpha=0$, we mean the usual $C^k$ space.
\item[2.] By $C^{a-}$, we mean the intersection of the H\"older spaces $C^{a'}$ for all $a'<a$.
\item[3.] The ball $B_1$ refers to the unit ball centered at the origin. Throughout this paper, we assume all balls to be centered at the origin unless specified otherwise. 
\end{itemize}

\begin{definition}[Partial convexity]
    We say that $\psi=\psi(x,s,p)\in C^2(B_1\times\re\times\re^n)$ satisfies partial convexity if 
\eqal{
\label{convex}
p\mapsto \psi(x,s,p)\text{ is convex, for each fixed }(x,s).
}
\end{definition}

Our main results in this paper are a classification of regularity for convex solutions of mean curvature potential equations, for both smooth phases and H\"older phases.

\begin{theorem}
\label{thm1}
Let $u$ be a convex viscosity solution of \eqref{slag} on $B_1\subset\mathbb R^n$.  If $\psi\in C^2(B_1\times\re\times\re^n)$ satisfies the partial convexity condition  \eqref{convex}
then $u\in C^{4-}(B_1)$.
\end{theorem}

\begin{theorem}
\label{thm2}
Let $u$ be a convex viscosity solution of \eqref{slag} on $B_1\subset\mathbb R^n$.  If $\psi\in C^{\alpha}(B_1\times\re\times\re^n)$ for some 
$\alpha\in(0,2]$, and $u\in C^{1,\beta}(B_1)$ for some $\beta\in (\frac{1}{1+\alpha},1)$, then $u\in C^{2+\alpha-}(B_1)$.
\end{theorem}

\begin{remark}
In particular, Theorem \ref{thm1} shows convex solutions of mean curvature flow potential equations \eqref{self}, \eqref{tran}, and \eqref{rotator} for $a\ge 0$ are smooth.  Theorem \ref{thm2} yields regularity for the $a<0$ case in rotator equation (\ref{rotator}) provided the graph $(x,Du(x))$ is $C^{1/3+}$.
\end{remark}
\begin{remark}
Using singular solutions, we see most of the conditions are optimal:
\begin{enumerate}
\item[1.2.1.]  Partial convexity \eqref{convex} is 
optimal, 
and the exponent range $\beta>1/(1+\al)$ or $\al>\beta^{-1}-1$ in Theorem \ref{thm2} is sharp.  Observe that for $\beta>0$,  $u(x)=|x|^{1+\beta}/(1+\beta)$ is a viscosity solution of a mean curvature equation \eqref{LMCp} with non-convex phase
\eqal{
\sum_{i=1}^n\arctan\lambda_i=\frac{n\pi}{2}-(n-1)\arctan(|Du|^{\beta^{-1}-1})-\arctan(\beta^{-1}|Du|^{\beta^{-1}-1}),
}
even when the phase is smooth $\beta=\frac{1}{3},\frac{1}{5},\frac{1}{7},\dots$.  Similarly, in Section \ref{sec:rot}, we prove that for $a<0$ and $c=n\pi/2$, the smooth, non-convex rotator equation \eqref{rotator} has a ``Yin-Yang" curve-type solution which is convex and $C^{1,1/3}$ near the origin.

\medskip

\item[1.2.2.] The convexity of $u$ in Theorem \ref{thm1} is 
sharp, by semi-convex $C^{1,1/3}$ solutions of special Lagrangian equation \eqref{s1} in dimension $n=3$ as shown by Nadirashvili-Vl\u{a}du\c{t} in \cite{NV}; see also Wang-Yuan \cite{WangY} for $C^{1,\beta}$ solutions, where $\beta=\frac{1}{3},\frac{1}{5},\frac{1}{7},\dots$.  In Section \ref{sec:self}, we show a similar solution exists to the shrinker/expander potential equation \eqref{self}.  
In addition, these special Lagrangian singular solutions $v(x_1,x_2,x_3)$ in three dimensions yield solutions of translator equation \eqref{tran} for $n\ge 5$.  Indeed, after a rotation, we can assume $k=k_0e_{n-1}+m_0e_n$ and $\ell=\ell_0 e_n$ for constants $\ell_0\neq 0$ and $k_0,m_0$.  First suppose $\ell\neq 0$.  Then we add a suitable quadratic polynomial $Ax_{n-1}^2/2+Bx_{n-1}x_n+Cx_n^2/2$ to $v(x_1,x_2,x_3)$.  Provided $\ell_0\neq 0$, we can find $A,B,C$, solve \eqref{tran}, and rotate back.  If instead $\ell=0$ and $k\neq 0$, we Legendre transform $(y,Dv)=(Du,x)$ to the $k=0$ and $\ell\neq 0$ case, repeat the above construction using the Legendre transform of $v$ in place of $v$, then Legendre transform back the result.  In dimension four, these singular solution constructions instead require $k$ and $\ell$ to be linearly dependent vectors in $\re^4$.
\end{enumerate}
\end{remark}

Each condition is also relevant to mean curvature flow.  A Lagrangian submanifold $M\subset \mbb C^n$ has compatible normal and almost complex structures, $TM\bot J(TM)$, and it was shown by Smoczyk \cite{Sm} that this condition is preserved by mean curvature flow. Smoczyk-Wang \cite{SmW} showed that the convexity condition of the initial potential $D^2u(x)\ge 0$ is preserved by Lagrangian mean curvature flow in the compact periodic setting.  In the general entire convex case, the preservation of convexity was reached in Chau-Chen-Yuan \cite{CCY13}. The convexity condition also allows solving the second boundary value problem in several cases \cite{BrW,WHB}; here, a convex solution of \eqref{LMC} is required to satisfy $Du(\Omega_1)=\Omega_2$ for prescribed uniformly convex, smooth domains $\Omega_i$; see also the construction by Huang \cite{H15} using Lagrangian mean curvature flow. Meanwhile, viscosity solutions are closed under mere uniform limits of the Lagrangian potential, see for example \cite[Proposition 2.9]{CC}.  Self-similar solutions \eqref{self} arise as Type 1 blowup limits (minimal blowup rate) in mean curvature flows, 
while eternal solutions, such as translators \eqref{tran} and rotators \eqref{rotator}, arise from Type 2 blowup limits (larger blowup rate).

\medskip
Closely related to convexity $D^2u\ge 0$ is the supercriticality condition $\theta(x)\ge (n-2)\pi/2$.  If either convexity $D^2u\ge 0$ or supercriticality is imposed, then the level set $\{M:tr\arctan(M)=\theta\}$ is convex, and fails convexity if otherwise, as shown by Yuan in \cite{YY0}. The stronger condition $\theta(x)\ge(n-1)\pi/2$ even implies $D^2u\ge 0$. Hessian estimates for constant critical and supercritical phase \eqref{s1} were established by Warren-Yuan \cite{WY9,WY} and Wang-Yuan \cite{WaY}; see also Li \cite{L}. By the Dirichlet problem solvability due to Caffarelli-Nirenberg-Spruck \cite{CNS} for constant critical or supercritical phases, these estimates yield interior regularity for \eqref{s1}. For $C^{1,1}$ supercritical phase $\psi=\psi(x)$,  Hessian estimates were proved in \cite {AB} for solutions of (\ref{slag}), and the Dirichlet problem was solved in \cite[Theorem 1.1]{AB1}.  The semi-convex $C^{1,1/3}$ solutions of Nadirashvili-Vl\u{a}du\c{t} \cite{NV} and Wang-Yuan \cite{WangY} have subcritical phases, so critical or supercritical phase is also an optimal condition.

\medskip
Unlike supercriticality, the convexity condition $D^2u\ge 0$ is unstable under smooth approximations of boundary data, which prevents using a priori estimates to establish regularity results.  It has been an open problem for how to derive interior regularity for convex viscosity solutions to \eqref{s1}, in light of the regularity for $W^{2,p}$ strong solutions by Bao-Chen \cite{BC} and the a priori estimate in Chen-Warren-Yuan \cite{WYJ}. 
This was only recently established in \cite{CSY} using the following low regularity approach.  If $u(x)$ were $C^2$, then the downward $U(n)$ rotation $\bar z=e^{-i\pi/4}z$ yields another Lagrangian submanifold $(\bar x,D\bar u(\bar x))$ whose potential $\bar u(\bar x)$ satisfies another special Lagrangian equation \eqref{s1} with smaller phase $c-n\pi/4$.  For convex Lipschitz $u(x)$, this is still true, but it requires generalizing the 
 transformation $e^{-i\pi/4}(x+iDu(x))=\bar x+iD\bar u(\bar x)$.  
Using the auxiliary functions $U(x):=su(x)+c|x|^2/2$ and $\bar U(\bar x):=-s\bar u(\bar x)+c|\bar x|^2/2$, where $s=\sin\pi/4,c=\cos\pi/4$, the $C^2$ case shows that $\bar U(\bar x)=U^*(\bar x)$, where $f^*(\bar x)=\sup_x(x\cdot\bar x-f(x))$ is the Legendre transform of strictly convex $f(x)$.  In the Lipschitz case, defining $\bar U(\bar x):=U^*(\bar x)$ supplies the desired rotation operation.  The concavity of the equation \eqref{s1} for convex $u$ then ensures the existence of subsolution approximations, guaranteeing that rotation preserves the solution property.  Because the gradient graph slope $0\le D^2u\le\infty$ drops to the Lipschitz bounds $-I\le D^2\bar u\le I$ under rotation,  methods from geometric measure theory give VMO regularity for $D^2\bar u$.  Applying perturbation theory for the now uniformly elliptic  PDE \eqref{s1} then yields $C^2$ and higher regularity for $\bar u(\bar x)$.  At this point, $(x,Du(x))\subset \mbb C^n$ is regular as a submanifold, but need not be a graph.  A geometric calculation shows that the rotated eigenvalue $\sqrt{1+\bar\la_{max}^2}$ is a subsolution of $\Delta_{\bar g}$, so the strong maximum principle guarantees $\bar\la_{max}<1$ everywhere, then $\la_{max}<\infty$.  In our previous paper \cite{BS1}, this approach was extended to $\psi=\psi(x)\in C^{2,\al}$ at the cost of a more complicated geometric calculation.  In passing, we also mention the Hessian estimates for convex smooth solutions with $C^{1,1}$ phase $\psi=\psi(x)$ in Warren \cite[Theorem 8]{WTh}.

\medskip
Our proof of Theorem \ref{thm1} extends the method of \cite{CSY} and \cite{BS1} to $\psi=\psi(x,u,Du)$, simplifies the final step of the proof by connecting it to the constant rank theorem of Caffarelli-Guan-Ma \cite{CGM}, and along the way, generalizes the constant rank theorem of Bian-Guan \cite{BG}.  Our insight is that the final step of showing $\bar\la_{max}<1$ is equivalent to showing that $\bar U(\bar x)$ is strongly convex, $D^2\bar U>0$.  Once shown, then $D^2U<\infty$ by Legendre duality $DU^*\circ DU(x)=x$.  To establish strong convexity, we observe that by convexity condition \eqref{convex}, the original equation $F(D^2U,DU,U,x)=0$ \eqref{slag} for $U(x)$ is concave if $u(x)$ is convex, so that for $U^*(\bar x)$ is inverse convex.  
The dependence of $\psi$ on $u$ corresponds to the dependence of the equation for $\bar U$ on Legendre transform $\bar x\cdot D\bar U-\bar U$.  It turns out the constant rank theorem of Bian-Guan works just as well in this generalized setting, as proved in Section \ref{sec:rank}, even though their result required additional convexity.  Overall, this can be thought of as a dual strategy to that for Monge-Amp\`ere \eqref{MA}, where instead strict convexity of $u(x)$ is key to regularity.

\medskip
The proof of Theorem \ref{thm2} requires a new low regularity approach to showing strong convexity of $U^*(\bar x)$.  In this H\"older $\psi$ setting, we only know $U^*(\bar x)\in C^{2,\al}$, and the constant rank theorem appears inapplicable.  The Monge-Amp\`ere approach to strict convexity in Urbas \cite{U} uses the assumed $C^{1,\al}$ or $W^{2,p}$ regularity and the specific determinant structure of the Monge-Amp\`ere equation to construct barriers, which lift the solution along any cylinder.  Our approach uses the arctangent structure in \eqref{slag} only to deduce that $\bar u(\bar x)\in C^{2,\al}$, as above.  Then more generally, if $f(x)\in C^{1,\beta}$ and $f^*(\bar x)\in C^{2,\al}$, we prove in Theorem \ref{thm:convex} that $f^*(\bar x)$ is strongly convex, hence $f(x)\in C^{1,1}$.  The heuristic idea uses the smoothness reversal of the Legendre transform, interpreted as gradient graph reflection $(x,D f(x))=(D f^*(\bar x),\bar x)$.  If $(x,Df(x))$ is flat near $x=0$, then $(\bar x,Df^*(\bar x))$ is steep near $\bar x=0$.  
Quantitatively, if a singular point of $f(x)$ is 
\textit{no steeper} than $\frac{1}{p}|x|^{p}$ for $p=1+\beta$, then $f^*(\bar x)$ is \textit{no flatter} than Legendre transform $\frac{1}{q}|\bar x|^q$, where $\frac{1}{p}+\frac{1}{q}=1$ and $q=1+\beta^{-1}$.  
This contradicts $f^*(\bar x)\in C^{2,\al}$, 
which implies that $f^*(\bar x)$ is \textit{at least as flat} as $|\bar x|^{2+\alpha}$ at a point failing strong convexity, if $2+\alpha>1+\beta^{-1}$.  

\medskip
The structure of the paper in proving Theorems \ref{thm1} and \ref{thm2} is as follows.  In Section \ref{sec:reg_rot}, we establish regularity for $\bar u(\bar x)$ in both situations of Theorems \ref{thm1} and \ref{thm2}, adapting the above approach of \cite{CSY} and \cite{BS1}.  Modifications are needed in the $\psi=\psi(x,u,Du)$ case in showing subsolution preservation (Proposition \ref{prop:sub}), and in the regularity theory for uniformly elliptic PDEs $F(D^2u)=f(x)$ with H\"older $f(x)$ (Theorem \ref{thm:regF}).  In Section \ref{sec:rank}, we prove Theorem \ref{thm:rank}, which generalizes the constant rank theorem of Bian-Guan \cite{BG}.  In Section \ref{sec:convex}, we prove Theorem \ref{thm:convex}, which establishes ``dual" strong convexity in lower regularity situations.  We conclude the proofs of Theorems \ref{thm1} and \ref{thm2} in Section \ref{sec:proofs}.  In Sections \ref{sec:self} and \ref{sec:rot}, we construct singular solutions for self-similar Lagrangian mean curvature flows \eqref{self} and Lagrangian rotator equation \eqref{rotator}, respectively.

\bigskip
\textbf{Acknowledgments.} The authors are grateful to Yu Yuan for helpful discussions. RS was partially supported by the NSF Graduate Research Fellowship Program under grant No. DGE-1762114.  We also thank the anonymous referees for many helpful suggestions.

\section{Regularity of the rotated Lagrangian potential}
\label{sec:reg_rot}

In this section, we first outline the proof of some preliminary results, adapting methods from \cite{CSY,BS1}, followed by regularity theory that we develop for $C^{1,1}$ viscosity solutions with VMO\footnote{[Vanishing mean oscillation]
Let $\Omega\subset\mathbb{R}^n$. A locally integrable function $v$ is in $VMO(\Omega)$ with modulus $\omega(r,\Omega)$ if
\[\omega(r,\Omega)=\sup_{x_0\in \Omega,0<r\leq R}\frac{1}{|{B_r(x_0)\cap \Omega}|} \int_{B_r(x_0)\cap \Omega} |v(x)-v_{x_0,r}|\rightarrow 0, \text{ as $r\rightarrow 0$}
\]
where $v_{x_0,r}$ is the average of $v$ over $B_r(x_0)\cap \Omega.$
} Hessian, modifying the proof in \cite{Y}. See Section \ref{sec:proof_reg_rot} below for definitions of rotated potential $\bar u(\bar x)$ and rotated domain $\pd\tilde u(B_1)$.
\subsection{$C^{2+\al}$ regularity on the rotated domain.}
\label{sec:proof_reg_rot}

We employ the notation described after Theorems \ref{thm1} and \ref{thm2} to state the following theorem.
\begin{theorem}
\label{thm:reg_rot}
Let $u(x)$ be a convex viscosity solution to \eqref{slag} on $B_1$ for $\psi\in C^\al(B_1\times\re\times\re^n)$ for some 
$\al\in(0,2]$.  If either $u\in C^1$ or $\psi$ satisfies partial convexity \eqref{convex}, then $\bar u(\bar x)\in C^{2+\al-}$ on $\pd\tilde u(B_1)$.
\end{theorem}

\begin{proof}
The proof is explained in the following two subsections. 
\subsubsection{\textbf{Lewy-Yuan rotation $\bar u(\bar x)$ solves a better equation}}
We first recall the Lewy-Yuan rotation by the angle of $\pi/4$, using the formulation developed in \cite{CSY} for Lipschitz potentials.  If $u(x)$ is convex but not $C^1$, then although its gradient graph $(x,Du(x))$ no longer makes sense, it still has a subdifferential $(x,\pd u(x))$, the slopes $y\in\pd u(x)$ of planes touching $u$ from below at $x$.  This need not be a graph, but it is isometric to a graph after a downward rotation of $z=(x,y)$ by $\pi/4$; 
see \cite[Proposition 1.1]{AA}.  
Namely, there exists a  $C^{1,1}$ potential $\bar u(\bar x)$ such that
\eqal{
\label{rotx}
\bar x&=cx+sy,\\
D\bar u(\bar x)&=-sx+cy,
}
whenever $y\in\pd u(x)$, where $c=\cos\pi/4,s=\sin\pi/4$.  Whenever $D^2u(x)$ exists, it transforms according to 
\begin{align}
\label{rule}
D^2\bar u(\bar x)=\left(\frac{-1+t}{1+t}\right)\circ D^2u(x).
\end{align}
Equivalently, the angles of the Hessian decrease by $\pi/4$:
\begin{align}
\label{angle}
\arctan\lambda_i(D^2\bar u(\bar x))=\arctan\lambda_i(D^2u(x))-\pi/4.
\end{align}
Accordingly, the new potential satisfies the Hessian bounds
\begin{align}
\label{bounds}
-I\le D^2\bar u(\bar x)\le I.
\end{align}
The actual potential is constructed using the Legendre transform:
\begin{align*}
f^*(\bar x)=\sup_{x}\left[x\cdot\bar x-f(x)\right].
\end{align*}
Provided $f$ is $C^1$ and strictly convex, the maximizer $x$ satisfies $\bar x=Df(x)$, or $\bar x\in \pd f(x)$ more generally.  
Letting $\tilde u(x)=su(x)+\frac{c}{2}|x|^2$, the new potential is given by 
(see \cite[Proposition 2.1]{CSY})
\begin{align}
\label{rot}
\bar u(\bar x)=\frac{c}{2s}|\bar x|^2-\frac{1}{s}(\tilde u)^*(\bar x).
\end{align}
The maximizer $x$ in the Legendre transform satisfies $\bar x\in\pd\tilde u(x)$.  Because $\tilde u(x)$ is strictly convex, the domain $\bar\Omega=\pd\tilde u(\Omega)$ of the rotated potential $\bar u(\bar x)$ is open and connected (\cite[Lemma 3.1]{CSY}).  

\medskip
Now by \eqref{rotx} and \eqref{angle}, it follows that $\bar u(\bar x)$ satisfies another mean curvature equation for almost every $x$ in $\Omega$:
\begin{align}
\label{snew}
\sum_{i=1}^n\arctan\lambda_i(D^2\bar u(\bar x))=\bar\psi(\bar x,\bar u(\bar x),D\bar u(\bar x)),
\end{align}
where the transformed phase can be written, in a non-essential way, in terms of the original one explicitly by
\begin{align*}
&\bar\psi(\bar x,\bar u,\bar p)=\psi(x(\bar x,\bar p),u(\bar x,x(\bar x,\bar p),\bar u),Du(\bar x,\bar p))-n\pi/4,\\
&x(\bar x,\bar p)=c\bar x-sD\bar u(\bar x),\\
&D u(\bar x,\bar p)=s\bar x+cD\bar u(\bar x),\\
&u(\bar x,x,\bar u)=\bar u+\frac{x\cdot\bar x}{s}-\frac{c}{2s}|x|^2-\frac{c}{2s}|\bar x|^2;
\end{align*}
see also \cite[p.334-335]{CW} for a similar formula.  However, it does not follow from this that \eqref{snew} holds for almost every $\bar x$ in $\bar\Omega$.  This difficulty was circumvented in \cite{CSY} for $\psi=const$ and in \cite{BS1} for continuous $\psi(x)$, using the additional concavity of $\arctan\lambda$ granted by $\lambda\ge 0$.  We now extend this to the present setting.

\begin{proposition}
\label{prop:sub}
Let $u$ be a convex viscosity solution of \eqref{slag} on $B_1$ for $\psi\in C(B_1\times\re\times\re^n)$ continuous.  If either $u\in C^1$ or $\psi$ satisfies partial convexity \eqref{convex}, then $\bar u(\bar x)$ is a viscosity solution of \eqref{snew} on $\pd \tilde u(B_1)$.
\end{proposition}

\medskip
The proofs follow \cite[Propositions 3.2 and 3.3]{CSY} and are reproduced here to account for the $Du$ dependence in $\psi$.  The outline of the subsolution proof is (i) show that smooth approximations of subsolutions are approximate subsolutions if $\psi$ has the convexity property \eqref{convex}, (ii) rotate the now-smooth approximate subsolutions to get approximate subsolutions of the smaller-angle equation \eqref{snew}, (iii) take a limit and use the closedness of viscosity subsolutions.  The supersolution case is more direct and uses the order-preserving nature of rotations to compare quadratic polynomials to supersolutions in the pre-rotated space.

\begin{proof}
First, if convex $u\in C^1$, then $x\mapsto \psi(x,u(x),Du(x))=:\phi(x)$ is continuous, and $u$ is a viscosity solution of $\theta(D^2u)=\phi(x)$.  Since $\phi(x)$ is trivially partial convex in the $Du$ variable, we need only consider the second case.

\smallskip
We next consider the case where $u$ is convex and $\psi$ satisfies \eqref{convex}.  It suffices to show that $\bar u(\bar x)$ is a viscosity 
solution 
of \eqref{snew} on $\pd \tilde u(B_{0.8})$, say, since the argument can be localized.  

\smallskip
We first show that $\bar u(\bar x)$ is a viscosity subsolution of \eqref{snew} on $\pd\tilde u(B_{0.8})$.  
First, we define continuous $\phi(x,p)=\psi(x,u(x),p)$, and note that $u(x)$ sub-solves $\theta(D^2u)=\phi(x,Du(x))$.  Following the argument in \cite[Theorems 5.5, 5.8]{CC} for the $F(D^2u)=0$ case, to show that averages are also subsolutions, we define a convex combination $v(x)=\sum_{i=1}^k t_iu(x+\e y_i)$ on $B_{0.9}$ for $\e>0$ small enough and some $y_i\in B_1$ and nonnegative $t_i:\sum_1^kt_i=1$.  Then, suppose quadratic $Q(x)$ touches $v(x)$ at the origin $x=0$ from above near $x=0$, say. 
We form an auxiliary function
\eqal{
w(x):=Q(x)-v(x)+\frac{\delta}{2}(|x|^2-r^2)
}
for small fixed $\delta>0$.  Note that $\partial w$ does not depend on $r$.  Since $Q-v\ge 0$, it follows that $w$ achieves 
non-negative values on small spheres $\pd B_r$, and 
$w(0)=-\delta r^2/2$ is strictly negative. 
Moreover, since $v$ is convex, $w_-=\min(w,0)$ is $C^{1,1}$ from above in $B_{r}$.  Hence, letting $\Gamma_w(x)$ be the convex envelope below $w_-$
in $B_{2r}$
, it follows from \cite[Lemma 3.5]{CC} that $\Gamma_w\in C^{1,1}$ in $B_{2r}$ and $\det D^2\Gamma_w(x)=0$ outside the contact set $\{w=\Gamma_w\}\cap B_r$, such that the ABP estimate holds:
\eqal{
0<\int_{\{w=\Gamma_w\}\cap B_r}\det D^2\Gamma_w(x)dx.
}
Choosing $z$ in the integration set such that $D^2u(z+\e y_i)$ exists for each $i$, the convexity of $\Gamma_w$ shows that $D^2Q+\delta I\ge D^2v(z)$.  We thus start with
\eqal{
\theta(D^2Q+\delta I)-\phi(0,DQ(0))\ge \theta(D^2v(z))-\phi(0,DQ(0)).
}
The concavity of arctangent operator $\theta(D^2v)$ for convex $v$ yields
\eqal{
\theta(D^2Q+\delta I)-\phi(0,DQ(0))&\ge \sum_{i=1}^k t_i\theta(D^2u(z+\e y_i))-\phi(0,DQ(0))\\
&=\sum_{i=1}^kt_i\phi(z+\e y_i,Du(z+\e y_i))-\phi(0,DQ(0)).
}
For the second term, we start with
\eqal{
DQ(0)=DQ(z)+O(r)=Dv(z)+Dw(z)+O(r).
}
Since $w(x)$ is semi-concave, it has continuous superdifferential by \cite[Corollary 24.5.1]{R}, which means $Dw(z)-Dw(0)$ vanishes as $r\to 0$. 
Note that a semi-convex function touched from above by a polynomial has a unique gradient.  Since $Dw(0)=0$, it follows $Dw(z)=o_r(1)$ as $r\to 0$.  We apply partial convexity \eqref{convex} to obtain
\eqal{
\theta(D^2Q+\delta I)-\phi(0,DQ(0))&\ge \sum_{i=1}^kt_i\phi(z+\e y_i,Du(z+\e y_i))-\phi(0,Dv(z))-o_r(1)\\
&\ge \sum_{i=1}^kt_i\left(\phi(z+\e y_i,Du(z+\e y_i))-\phi(0,Du(z+\e y_i))\right)-o_r(1)\\
&\ge -o_\e(1)-o_r(1),
}
where, by the local Lipschitz bound for convex $u$, $o_\e(1)$ vanishes uniformly in $B_{0.9}$ as $\e\to 0$.  Sending $r\to 0$ and $\delta\to 0$ shows that average $v$ is a subsolution of the equation
\eqal{
\label{eqapprox}
\theta(D^2v)\ge \phi(x,Dv(x))-o_\e(1),\qquad x\in B_{0.9}(0).
}
More to the point, letting $\eta_\e=\e^{-n}\eta(x/\e)$ be a standard nonnegative mollifier with $\int \eta=1$, $supp\,\eta\subset B_1$, we can uniformly approximate the solid average $u_\e(x):=\int_{B_1}u(x+\e y)\eta(y)dy$ by some such $v(x)$ as above.  It follows from the closedness 
of subsolutions \cite[Remark 6.3]{CIL} under uniformly convergent subsolutions that $u_\e(x)$ is a smooth subsolution of \eqref{eqapprox} on $B_{0.9}(0)$.

\smallskip
Next, it was shown in \cite[Proposition 3.3]{CSY} the subdifferential 
$\pd\tilde u_\e\subset \re^n\times\re^n$
converges to $\pd\tilde u$.  In particular, the rotation
$\bar u_\e=(u_\e)^-$ 
is well defined on the open set $\bar\Omega:=\pd \tilde u(B_{0.8})\subset\pd\tilde u_\e(B_{0.9})$ for small $\e$.  By the gradient and Hessian transformation rules \eqref{rotx} and \eqref{angle}, it follows that $\bar u_\e$ is a subsolution of the approximate rotated equation,
\eqal{
\theta(D^2\bar u_\e)\ge \phi(c\bar x-sD\bar u_\e(\bar x),s\bar x+cD\bar u_\e(\bar x))-n\pi/4-o_\e(1),\qquad \bar x\in\bar\Omega.
}
Note that $o_\e(1)$ is still uniformly small.  We now send $\e\to 0$. As in \cite{CSY}, since the Legendre transform is order reversing, the rotation \eqref{rot} is order preserving and respects uniform convergence, $\bar u_\e\to\bar u$ on $\bar\Omega$.  The closedness \cite[Proposition 2.9]{CC} of viscosity subsolutions under uniformly convergent solutions and locally uniformly convergent PDE operators shows that $\bar u$ is a subsolution of rotated equation \eqref{snew}.

\medskip

We now show $\bar u(\bar x)$ is a supersolution of \eqref{snew} on $\bar\Omega:=\pd\tilde u(B_{0.8})$, reproducing the argument in \cite[Proposition 3.2]{CSY}.  Let $\bar Q$ be a quadratic polynomial touching $\bar u$ from below nearby $\bar x_0\in \partial \tilde u(x_0)$, i.e. $\bar Q(\bar x_0)=\bar u(\bar x_0)$ and $\bar Q\le \bar u$ in an open subset of open set $\bar\Omega$.  After shifting $x$, subtracting a linear function from $u(x)$, and modifying $\psi(x,u,Du)$ accordingly, we assume that $x_0=\bar x_0=0$ and $u(0)=\bar u(0)=0$.  By \eqref{bounds}, we know $D^2\bar Q\le D^2\bar u\le I$.  Because we could subtract $\epsilon|\bar x|^2$ from $\bar Q$, then send $\epsilon\to 0$, we assume $D^2\bar Q<I$.  By \eqref{rule}, this implies that the pre-rotated function $Q(x)$ is also a quadratic polynomial, i.e. $D^2Q<\infty$.  The subtraction also guarantees that $\bar Q<\bar u$ on an open neighborhood of closed set $\pd\tilde u(0)$.  Indeed, we have $(\tilde u)^*(\bar x)\equiv 0$ for all $\bar x\in \pd\tilde u(0)$, so $\bar Q\le \bar u$ even on the boundary of $\pd\tilde u(0)$.  By the order preservation of rotation and the continuity of the subgradient mapping $\pd\tilde u:\re\to\re^n\times\re^n$ (see \cite[Corollary 24.5.1]{R}), we find that pre-rotation $Q$ touches $u$ from below at $x=0$ on an open neighborhood of $0\in B_1(0)$.  Because $u$ is a supersolution, and $Q$ is a quadratic polynomial, we deduce
$$
\sum_{i=1}^n\arctan\lambda_i(D^2Q)\le \psi(0,0,0).
$$
Using the transformation rules \eqref{rotx} and \eqref{angle} for rotation, we conclude
$$
\sum_{i=1}^n\arctan\lambda_i(D^2\bar Q)\le \psi(0,0,0)-n\pi/4,
$$
which is \eqref{snew} after the normalization.  This verifies that $\bar u$ is a supersolution.
\end{proof}

\subsubsection{\textbf{Rotated regularity $\bar u(\bar x)\in C^{2+\alpha-}$}}

Now that $\bar u(\bar x)$ is $C^{1,1}$, the mean curvature equation \eqref{snew} is uniformly elliptic with a continuous right-hand side 
as a function of $\bar x$.


\medskip
We next claim $D^2\bar u$ is VMO, even if $\psi\in C^\alpha$.  The mean curvature is no longer bounded, and there is no monotonicity formula.  Mere continuity is enough for VMO, and a simple fix is as follows.  Repeating the proof of \cite[Proposition 3.1]{BS1} verbatim, we blowup the solution using a quadratic rescaling.  The $W^{2,\epsilon}$ estimate in \cite[Proposition 7.4]{CC}, combined with the Hessian bounds \eqref{bounds}, give local $W^{2,p}$ and $C^{1,\alpha}$ convergence of a subsequence.  The limit is a $C^{1,1}$ viscosity solution of the constant phase, special Lagrangian equation on $\re^n$ with Hessian bounds \eqref{bounds}, so the work of Yuan from the 2000's now applies \cite{Y}; we refer verbatim to \cite[pg. 9, Step 1]{CSY} which shows that the limit is smooth, and to \cite[pg. 122, Step B]{YY} which shows that a smooth entire solution of the constant phase equation with bounds \eqref{bounds} is quadratic.  Using this rigidity of the limits, we conclude from the proof of \cite[Proposition 2.3]{Y} verbatim that the Hessian of the solution is VMO.

\medskip
Now that $D^2\bar u$ is in VMO, we prepare to invoke the regularity theory in Section \ref{sec:reg}.  Let $G(D^2\bar u)$ be a $C^2$ uniformly elliptic operator which extends the truncation of $F(D^2\bar u)$ in \eqref{snew} to $|D^2\bar u|\le 1$.  By the Hessian bounds \eqref{bounds}, it follows that $\bar u$ is also a viscosity solution of
$$
G(D^2\bar u)=g(\bar x),
$$
where $\bar x\mapsto\bar\psi(\bar x,\bar u(\bar x),D\bar u(\bar x))=:g(\bar x)$ 
is a $C^{0,\min(\al,1)}$ right hand side, since $D\bar u(\bar x)$ is Lipschitz.  If $\al\in(0,1)$, we apply Theorem \ref{thm:regF} and deduce that $\bar u\in C^{2,\al}=C^{2+\al}$.  If $\al=1$, then we apply the previous result for any $\al'<1$ and deduce $\bar u\in C^{3-}$.  If $\al\in(1,2)$, then 
$\bar u\in C^{3-}$ to start with.  Now that $D\bar u(\bar x)\in C^{2-}$, we infer that $g\in C^{\al}$, so the linear equation satisfied by the difference quotient $[\bar u(\bar x+h)-\bar u(\bar x)]/|h|$ has a $C^{\al-1}$ right-hand side, and the associated second order linear operator involving $D^2\bar u\in C^{1-}$ has $C^{1-}$ coefficients.  It follows from Schauder estimates that $D\bar u(\bar x)\in C^{2,\al-1}$, hence $\bar u\in C^{3,\al-1}=C^{2+\al}$.  If $\alpha=2$, we apply the previous result for any $\al'<2$ and deduce that $\bar u\in C^{4-}$.  This concludes the proof of Theorem \ref{thm:reg_rot}.

\end{proof}
\subsection{$C^{2,\alpha}$ estimates for VMO solutions of fully non-linear elliptic equations}
\label{sec:reg}

In this section, we develop regularity theory for $C^{1,1}$ viscosity solutions of fully non-linear uniformly elliptic equations of the form
\begin{equation}
    F(D^2u)=f(x) \label{F} \text{  in  }B_1
\end{equation}
where $F\in C^2$, $f\in C^\alpha$ is H\"older continuous with $\alpha\in (0,1)$ (for this section only), and the Hessian of $u$ is in VMO. We assume $|D^2u|_{L^{\infty}(B_1)}\leq K$. Since the ellipticity constants of the uniformly elliptic Lagrangian mean curvature type equations depend on $K$, we
simplify notation and denote by $K$ a large constant controlling the Hessian, the ellipticity constants, and the VMO modulus.

\begin{theorem}
\label{thm:regF}
 Suppose that $u$ is a $ C^{1,1}$ viscosity solution of (\ref{F}) on
$B_1$. If $f\in C^{\alpha}(B_1)$ and $D^2u\in VMO(B_{3/4})$, then $u\in C^{2,\alpha}(B_{1/2})$ and satisfies
the following estimate 
\begin{equation}
||u||_{C^{2,\alpha}(B_{1/2})}\leq C(n, K,\alpha,\left\Vert f\right\Vert _{C^{\alpha}(B_{1})}) .\label{AAA}%
\end{equation}

\end{theorem}
The following Lemma is an extension of \cite[Lemma 3.8]{CC1} to a variable right hand side. Note that this Lemma holds good under weaker assumptions on equation (\ref{F}) as compared to the above Theorem.
\begin{lemma}\label{rrr}
Suppose that $u\in C(B_{r})$ is a viscosity solution of $G(D^2u)=g(x)$ in $B_{r}$, for some $r>0$, where $G$ is a uniformly elliptic operator and $g\in C^{\alpha}(B_{r})$. If $Q$ is a quadratic polynomial, then there exists a quadratic polynomial $P$ such that 
\begin{align*}
G(D^2P)=g(0),\\
    ||u-P||_{L^{\infty}(B_{r})}\leq C||u-Q||_{L^{\infty}(B_{r})}+C[g]_{C^{\alpha}(B_{r})}r^{2+\alpha}
\end{align*}
where $C$ is a universal constant. 
\end{lemma}

\begin{proof}
We denote the ellipticity constants of $G$ by $0<\lambda\leq\Lambda$. Let $a=G(D^2Q)$, $b=||u-Q||_{L^{\infty}(B_{r})}$, and $R(x)=2r^{-2}b|x|^2$. From \cite[Proposition 2.13]{CC}, we get \[ u-Q\in S(\lambda/n,\Lambda,G(D^2u)-G(D^2Q))=S(\lambda/n,\Lambda,g(x)-a). 
\] We observe the following
\begin{align*}
    [R-(u-Q)]|_{\partial B_{r}}=2b-(u-Q)|_{\partial B_{r}}\geq 2b-(u-Q)(0)-2||u-Q||_{L^{\infty}(B_{r})}
    =[R-(u-Q)](0).
\end{align*}
Since $u-Q$ is a viscosity subsolution and the minimum of $R-(u-Q)$ is achieved at an interior point of $B_{r}$, say $x_0$, we get $\mathcal{M}^+(D^2R,\lambda/n,\Lambda)\geq g(x_0)-a$. Similarly, using the fact that $u-Q$ is also a viscosity supersolution, we get $|g(x_0)-a|\leq Cr^{-2}b$. This shows
\[ |g(0)-G(D^2Q)|\leq Cr^{-2}||u-Q||_{L^{\infty}(B_{r})}+C[g]_{C^{\alpha}(B_{r})}r^\alpha
\]
By the ellipticity of $G$, there exists $s\in \mathbb R$ such that $G(D^2Q+s I_n)=g(0)$ and \[|s|\leq Cr^{-2}||u-Q||_{L^{\infty}(B_{r})}+C[g]_{C^{\alpha}(B_{r})}r^\alpha.\]Fixing this $s$ we define $P(x)=Q(x)+\frac{1}{2}s|x|^2$. Therefore, we get
\begin{align*}
    ||u-P||_{L^{\infty}(B_{r})}\leq ||u-Q||_{L^{\infty}(B_{r})}+||Q-P||_{L^{\infty}(B_{r})}\\
    =||u-Q||_{L^{\infty}(B_{r})}+\frac{1}{2}|s|r^2\leq C||u-Q||_{L^{\infty}(B_{r})}+C[g]_{C^{\alpha}(B_{r})}r^{2+\alpha}.
\end{align*}
\end{proof}

\begin{proposition}\label{p}
Let $u$ be a $C^{1,1}$ viscosity solution of (\ref{F}) on $B_{1}$, where $D^2u\in VMO(B_{3/4})$, with VMO modulus $\omega(r)$. Then for any $\varepsilon_0>0$ there exists $\eta=\eta(n,\omega, K,||f||_{C^{\alpha}(B_1)},\varepsilon_0)$ and a quadratic polynomial $P$ such that 
\begin{align*}
   \sup_{B_1} |\frac{u(\eta x)}{\eta^2}-P(x)|\leq \varepsilon_0 \\
    F(D^2P)=f(0).
\end{align*}
\end{proposition}
\begin{proof}
    For $r>0$ to be chosen later, let $w_r(x)=\frac{1}{r^2}u(rx)$ and \begin{align*}
        s_r=Tr\frac{1}{|B_1|}\int_{B_1}D^2w_r=Tr\frac{1}{|B_r|}\int_{B_r}D^2u.
    \end{align*} Note that $|s_r|\leq n|D^2u|_{L^{\infty}(B_1)}=nK.$ We solve the boundary value problem 
\begin{align}
\Delta h(x)   =s_r\text{ in }B_{1}\nonumber\\
h(x)   =w_r\text{ on }\partial{B_{1}}. \nonumber
\end{align} Using the Alexandrov-Bakelman-Pucci maximum principle \cite[Theorem 3.2]{CC}, we observe the following
\begin{align*}
    ||w_r-h||_{L^{\infty}}(B_1)\leq C(n)||\Delta w_r-\Delta h||_{L^{\infty}}(B_1)\\
    \leq C(n,K)[\int_{B_1}|D^2u(rx)-(D^2u)_{0,r}|]^{1/n}\\
    =C(n,K)[\frac{1}{|B_r|}\int_{B_r}|D^2u(x)-(D^2u)_{0,r}|]^{1/n}\leq C(n,K)\omega^{1/n}(r).
\end{align*}
Note that $h$ satisfies the following estimate
\begin{align*}
    ||D^3h||_{L^{\infty}}(B_{1/2})\leq C(n)\sup_{\partial B_1}|w_r(x)-w_r(0)-\langle \nabla w_r(0),x \rangle -\frac{1}{2}s_r|x|^2|\leq C(n,K).
\end{align*}
We define the quadratic part of $h$ at the origin to be $\bar P$.  By Taylor approximation, we obtain
\begin{align*}
    |w_r(x)-\bar P(x)|\leq C(n,K)\omega^{1/n}(r)+C(n,K)|x|^3.
\end{align*}
Since $w_r$ satisfies $F(D^2\omega_r)=f(rx)$, on using Lemma \ref{rrr}, we see that there exists a quadratic polynomial $\bar{\bar{P}}(x)$ such that $F(D^2\bar{\bar{P}})=f(0)$ and 
\[ |w_r(x)-\bar{\bar{P}}(x)| \leq C(n,K)\omega^{1/n}(r)+C(n,K)|x|^3+||f(rx)||_{C^{\alpha}(B_1)}r^{2+\alpha}.
\]
For $\rho>0$ to be chosen later, let $x=\rho y$ and $P(y)=\frac{1}{\rho^2}\bar{\bar{P}}(\rho y)$. Given any $\varepsilon_0>0$, we choose $r,\rho$ depending on $n,K,\omega,||f||_{C^{\alpha}(B_1)},\varepsilon_0$ such that for $\eta=\eta (n,K,\omega,||f||_{C^{\alpha}(B_1)},\varepsilon_0)=\rho r$ we get
\[|\frac{1}{\eta^2}u(\eta y)-P(y)|\leq \varepsilon_0.
\]
\end{proof}

\begin{lemma}
\label{lemma 2.6}There exists $\delta>0$ depending on $n,\alpha,K$ such that if $u$ is a $C^{1,1}$ viscosity solution of (\ref{F}) in
$B_{1}$ with $|D^2u|\le K$ and 
$[f]_{C^{\alpha}(B_{1})}\leq\delta$,
then there exists a quadratic polynomial $P$ such that
\begin{align}
||u-P||_{L^{\infty}(B_{r})}    \leq C_{0}r^{2+\alpha}\hspace{0.2cm},\forall
r\leq1,\nonumber\\
|DP(0)|+||D^{2}P||    \leq C_{0} \label{P1}%
\end{align}
for some constant $C_{0}>0$ depending on $n,\alpha,K$.
\end{lemma}

\begin{proof}
The proof follows from the following claim.

\begin{claim}
\label{claim_CL3} 
Given the assumptions of the above Lemma, there exist $0<\mu<1,m$ depending on $n,K,\alpha$ such that if
\begin{align}
    ||u-P||_{L^{\infty}(B_1)}\leq \mu^{2+\alpha+m}\nonumber\\
    F(D^2P)=f(0) \label{a}
\end{align}
then we have a sequence $P_{k}(x)=a_{k}+b_{k}\cdot x+\frac{1}{2}x^{t}c_{k}\cdot x$ satisfying
\begin{align}
&F(D^{2}P_{k})    =f(0)\label{L1b}\\
&||u-P_{k}||_{L^{\infty}(B_{\mu^{k}})}   \leq\mu^{k(2+\alpha)+m}\label{L1a}\\
&|a_{k}-a_{k-1}|+\mu^{k-1}|b_{k}-b_{k-1}|+\mu^{2(k-1)}|c_{k}-c_{k-1}|    \leq
C(n,K)\mu^{(k-1)(2+\alpha)+m}. \label{L1}%
\end{align}

\end{claim}

\begin{proof}
Let $P_{1}=P$. Then for $k=1$, the claim holds trivially. For $\mu$ determined by (\ref{tin}), we show
that if (\ref{L1a}) holds for $k=i$, then it holds good for $k=i+1.$
We define
\begin{align}
v_{i}(x)    =\frac{(u-P_{i})(\mu^{i}x)}{\mu^{i(2+\alpha)+m}}\label{vedf}\\
F_{i}(N)    =\frac{F(\mu^{i\alpha+m}N+c_{i})}{\mu^{i\alpha+m}}\nonumber\\
f_{i}(x)    =\frac{f(\mu^{i}x)}{\mu^{i\alpha+m}}\nonumber
\end{align}
where $P_{i}(x)=a_{i}+b_{i}\cdot x+\frac{1}{2}x^{t}\cdot c_{i}%
x$. So far we have
\begin{align}
F_{i}(D^{2}v_{i}(x))=f_{i}(x) \text{ and } \left\Vert v_{i}\right\Vert _{L^{\infty}(B_{1})}\leq C(n,K).\label{V} 
\end{align}
 Denoting $\omega_{n}$ as the volume of a unit ball in $n$ dimensions we observe that 
\begin{align}
||f_{i}-f_i(0)||_{L^{n}(B_{1})}   =\mu^{-i\alpha-m}\mu^{-i}||f-f(0)||_{L^{n}(B_{\mu^{i}}%
)}
  \leq\mu^{-i\alpha-m}\mu^{-i}\left\vert B_{\mu^{i}}\right\vert ^{1/n}\delta
\mu^{i\alpha}=\mu^{-m}\omega_{n}  ^{1/n}\delta\nonumber.
\end{align}

Let $h$ be the solution of the boundary value problem where $0$ is in the Hessian space:
\begin{align}
DF_{i}(0)\cdot D^{2}h   =0\text{ in }B_{1/2}\nonumber\\
h   =v_{i}\hspace{0.4cm}\text{ on }\partial{B_{1/2}}. \nonumber
\end{align}
Next, we observe that 
\begin{align}
    F_{i}(D^2h)=F_i(0)+\sum_{a,b=1}^n (F_{i})_{ab}(0)h_{ab}+O(|D^2F||D^2h|^2\mu^{i\alpha})\nonumber\\
    =\frac{f(0)}{\mu^{i\alpha+m}}+O(|D^2F||D^2h|^2\mu^{i\alpha})\label{ref}.
\end{align}
Using H\"older estimates on $v_i$, with $\beta=\beta(n,K)$, from \cite[Proposition 4.10]{CC} we see that
\[||v_i||_{C^{\beta}(\overline{B_{1/4}})}\leq C(n,K)[1+\mu^{-m}\omega_n^{1/n}\delta ].
\]
By the global H\"older estimate on $h$, we have
\begin{align}
||h||_{C^{2,\beta/2}(\overline{B_{1/2}})}  
 \leq C(n,K)||v_{i}||_{C^{\beta}(\partial B_{1/2})}
 \leq C(n,K)(1+\mu^{-m}
 \delta ). \nonumber
\end{align}
 By the Alexandrov-Bakelman-Pucci maximum principle we have for $\varepsilon>0$ to be chosen later
\begin{align*}
    ||v_{i}-h||_{L^{\infty}(B_{1/2-\varepsilon})}&\leq \sup_{\partial B_{1/2-\varepsilon}}|v_i-h|+C(n,K)||F_i(D^2v_i)-F_i(D^2h)||_{L^{\infty}(B_{1/2-\varepsilon})}\\
    &\leq C(n,K)[(\varepsilon^\beta + \varepsilon^{\beta/2})(1+\mu^{-m}
    \delta )+|D^2 F|\varepsilon^{-4}\mu^{i\alpha+m}+\mu^{-m}
    \delta(1/2-\varepsilon)^\alpha]\\
    &\leq C(n,K)(\varepsilon^{\beta/2}+\mu^{i\alpha+m}\varepsilon^{-4}+\mu^{-m}\delta)
\end{align*}
where the second 
inequality follows from (\ref{ref}). Since $h\in C^{2,\alpha}$, there exists a quadratic polynomial $\bar{P}$ given by
\[
\bar{P}(x)=h(0)+Dh(0)\cdot x+\frac{1}{2}x^{t}D^{2}h(0)\cdot x
\]
such that
\begin{equation*}
||h-\bar{P}||_{L^{\infty}(B_{\mu})}\leq C(n,K)\mu^{3}. 
\end{equation*}
So far we have 
\begin{align*}
||v_{i}-\bar{P}||_{L^{\infty}(B_{\mu})} 
\leq & \,C(n,K)(\e^{\beta/2}+\mu^{i\alpha+m}\e^{-4}+\mu^{-m}\delta)+C(n,K)\mu^3.
\end{align*}
From Lemma \ref{rrr} we see that there exists a quadratic polynomial $\bar{\bar{P}}$ such that \\
$F_i(D^2\bar{\bar{P}})=f(0)$ and 
\begin{align*}
||v_{i}-\bar{\bar{P}}||_{L^{\infty}(B_{\mu})} 
&\leq C(n,K)(\e^{\beta/2}+\mu^{i\alpha+m}\e^{-4}+\mu^{-m}\delta+\mu^3)+C(n,K)\|f_i\|_{C^\al(B_\mu)}\mu^{2+\al}\\
&\leq C(n,K)(\e^{\beta/2}+\mu^{i\alpha+m}\e^{-4}+\mu^{-m}\delta+\mu^3+\mu^{-m}\delta\mu^{2+\al}).
\end{align*}
Choosing $\e^{\beta/2}$ and $\mu^3$, then $m$, then $\delta$ suitably, we get
\begin{align}
||v_{i}-\bar{\bar{P}}||_{L^{\infty}(B_{\mu})} 
\leq \mu^{2+\alpha}. \label{tin}
\end{align}
Rescaling the above back via (\ref{vedf}) we see that
\begin{equation}
|u(x)-P_{i}(x)-\mu^{i(2+\alpha)+m}\bar{\bar{P}}(\mu^{-i}x)|\leq\mu^{(2+\alpha)(i+1)+m} \text{ for all } x\in B_{\mu^{i+1}}.
\label{L9}%
\end{equation}
We define
\begin{equation}
P_{i+1}(x)=P_{i}(x)+\mu^{i(2+\alpha)+m}\bar{\bar{P}}(\mu^{-i}x) \label{H2}%
\end{equation}
and we have
\begin{equation}
c_{i+1}=c_{i}+\mu^{i\alpha+m}D^2\bar{\bar{P}}. \label{L5s}%
\end{equation}
From (\ref{L9}) we see that
$
||u-P_{i+1}||_{L^{\infty}(B_{\mu^{i+1}})}\leq\mu^{(i+1)(2+\alpha)+m}%
$
which proves (\ref{L1a}) for $k=i+1$, and from(\ref{L5s}) we get
$
F(c_{i+1})=f(0)$, thereby proving (\ref{L1b}). Evaluating (\ref{H2}) and its first and second
derivatives at $x=0$ we get
$
a_{i+1}-a_{i}    =\mu^{i(2+\alpha)+m}\bar{\bar P}(0),
b_{i+1}-b_{i}    =\mu^{i(1+\alpha)+m}D\bar{\bar P}(0),
c_{i+1}-c_{i}    =\mu^{i\alpha+m}D^{2}\bar{\bar P}(0).
$
Thus
\begin{align*}
 |a_{i+1}-a_{i}|+\mu^{i}| b_{i+1}-b_{i}|
+\mu^{2i}| c_{i+1}-c_{i}|\\
  =\mu^{i(2+\alpha)+m}(|\bar{\bar P}(0)|+| D\bar{\bar P}(0)| +|
D^{2}\bar{\bar P}(0)| )\\
  \leq\mu^{i(2+\alpha)+m}C(n,K)%
\end{align*}
proving (\ref{L1}). This proves claim \ref{claim_CL3}.
\end{proof}
 We observe that  
assumption (\ref{a}) in the above claim follows from a rescaling of the result in Proposition \ref{p} applied to the rescaled function $u(\eta x)/\eta^2$. The remaining proof of the Lemma follows from a standard iterative argument (see \cite{Y}).

\end{proof}

\begin{proof}
[Proof of Theorem \ref{thm:regF}] We may assume $u(0)=Du(0)=0$. We will first
prove that the estimate (\ref{AAA}) holds at the origin. We
show that there exists a quadratic polynomial $P$ such that
\begin{align}
||u-P||_{L^{\infty}(B_{r})}   \leq Cr^{2+\alpha}\hspace
{0.2cm},\forall r\leq1\nonumber\\
|DP(0)|+||D^{2}P||   \leq C \label{P2}%
\end{align}
where $C=C(n, K,\alpha,\left\Vert f\right\Vert _{C^{\alpha}(B_{1})})$.
The proof now follows directly from Lemma \ref{lemma 2.6}, if  for all $x\in B_{1}$ we rescale 
$
\tilde{u}(x)=\tau^{-2}u(\tau x), 
$
so that $\tilde u$ solves 
$$
F(D^2\tilde u(x))=f(x/\tau)=:f_\tau(x).
$$
The H\"older seminorm becomes $$
[f_\tau]_{C^\alpha(B_1)}\le \tau^{-\alpha}[f]_{C^\alpha(B_1)}\le\delta
$$
if $\tau^\alpha=[f]_{C^\alpha(B_1)}/\delta$ is large enough.  Therefore the equation $F(D^{2}\tilde{u}(x))=f_{\tau}(x)$
satisfies all the conditions of Lemma \ref{lemma 2.6} and hence the function
$\tilde{u}$ satisfies the estimates in (\ref{P1}). 
In particular, there exists
$\tilde{P}$ such that
\begin{align}
||\tilde{u}-\tilde{P}||_{L^{\infty}(B_{r})}    \leq C(n, K,\alpha,\left\Vert f\right\Vert _{C^{\alpha}(B_{1})})r^{2+\alpha},\text{
}\hspace{0.2cm}\forall r\leq1\nonumber\\
|D\tilde{P}(0)|+||D^{2}\tilde{P}||    \leq C(n, K,\alpha,\left\Vert f\right\Vert _{C^{\alpha}(B_{1})})
\nonumber
\end{align}
which proves (\ref{P2}) on substituting back to the original function. Similarly, we prove this holds at every point $x_0\in B_{1/2}$, which in turn proves (\ref{AAA}).
\end{proof}

\section{Generalizing the constant rank theorem}
\label{sec:rank}

In \cite{BG, CGM, SW}, there are versions of the following constant rank theorem. If $\Omega\subset\re^n$ and equation $G(D^2u,Du,u,x)=0$ is an elliptic $C^2$ operator on $Sym^+(n)\times\re^n\times\re\times\Omega$ satisfying ``inverse convexity"
\eqal{
\label{ic}
(A,s,x)\mapsto G(A^{-1},p,s,x)\text{ is locally convex for each fixed }p\in\re^n,
}
and $u$ is a $C^3$ solution, then the rank of the Hessian $D^2u(x)$ is constant in $\Omega$.  We generalize this to $G(\,.\,,\,v)=0$, where $v:=x\cdot Du-u$ 
can be interpreted as the pullback of the Legendre transform
$u^*(\bar x)=x\cdot\bar x-u(x)$ under the gradient mapping $x\mapsto Du$, or $v=u^*\circ Du$.  This new dependence is nontrivial, since unlike for the $x$ or $u$ variables, $G$ need not be convex in $v$.


\begin{theorem}
\label{thm:rank}
 Let $G(M,p,s,x,v)$ be $C^{2}$,  elliptic
 \eqal{
 G^{ab}:=\frac{\pd G}{\pd M_{ab}}(M,p,s,x,v)>0,
 }
 and inverse convex
 \eqal{
 \label{ic1}
 (A,s,x)\mapsto G(A^{-1},p,s,x,v)\text{ is locally convex for each fixed }(p,v)\in \re^n\times\re.
 }
 If $u$ is a $C^{3}$ convex solution of $G(D^2u,Du,u,x,x\cdot Du-u)=0$ in $B_1$, then $D^2u(x)$ has constant rank in $B_1$.
\end{theorem}

\begin{remark}
\label{rem:nou}
Defining an equation without extra $u$ dependence,
\eqal{
H(D^2u,Du,x,x\cdot Du-u)=G(D^2u,Du,u,x,x\cdot Du-u)
}
we see that $H$ is elliptic and inverse convex in the $A,x$ arguments, with no restrictions on the $p,v$ arguments.  In fact, Theorem \ref{thm:rank} holds if and only if the functionally dependent convex $u$ argument in $G$ is eliminated.
\end{remark}

\begin{proof}
To generalize the proof from \cite{SW}, it suffices to confirm there are nonnegative supersolutions of $G^{ab}\pd_{ab}$ which vanish according to the rank of $D^2u(x)$, since the Harnack inequality implies such vanishing occurs everywhere.  The rest of the proof in \cite{SW} involves a $W^{4,n+}$ approximation maneuver that ensures the eigenvalues of $D^2u(x)$ are distinct away from a closed null set and adapting the Harnack inequality to an approximate setting.  We refer therein and omit these details, which are the same in our generalized setting.

\medskip
Let $P(x)$ be a polynomial such that $D^2P(x)$ has distinct positive eigenvalues $0<\La_1<\cdots<\La_n$ near $x=x_0$.  We assume $D^2P(x_0)$ is diagonal with $\La_i=P_{ii}$.  Given integer $k\in[1,n]$, we put
\eqal{
Q^k=\La_k+2\La_{k-1}+\cdots+k\La_1=\sum_{j=1}^k(k+1-j)\La_j.
}
Our objective is to establish the following supersolution recursion estimate, pointwise at $x_0$:
\eqal{
\label{super}
G^{ab}Q^k_{ab}\le C\sum_{\ell=1}^k(Q^\ell+|DQ^\ell|+|\pd_{\ell\ell}G|),
}
where $C$ depends on the second order Taylor polynomial of $G$ at $P(x_0)$ and the third order Taylor polynomial of $P$ at $x_0$, and $\pd_{\ell}G$ is computed at $P(x)$.  In the case that $P$ approximates $u$ solving $G=0$ in $W^{4,n+}$, then the final term is small in $L^{n+}$.  This will verify the approximate supersolution inequality in \cite[equation (2.5)]{SW} for the case of $\pd_vG=0$, accounting for their induction hypothesis \cite[equation (2.4)]{SW} that $Q^\ell,DQ^\ell$ are small for $\ell<k$.

\medskip
A direct calculation is straightforward, as in \cite{SW}, and the idea is that any $v$ contributions directly absorb into the $Q^k$ and $DQ^k$ terms.  Instead, we will show that the $v$ dependence can be replaced at $x_0$ by $p$ dependence and affine $u,x$ dependence, then invoke the calculation from \cite{SW}.  This technique uses the pointwise nature of \eqref{super}.  The first two jets of $v=x\cdot DP-P$ at $x_0$ can be expressed in terms of $DP$ and affine $P$:
\eqal{
\label{jet}
&\pd_iv=x_0\cdot DP_i=\pd_i(x_0\cdot DP),\\
&\pd_{ij}v=x_0\cdot DP_{ij}+P_{ij}=\pd_{ij}(x_0\cdot DP)+\pd_{ij}P.
}
Accordingly, we define a new elliptic operator by replacing $v$ with a $p$-dependent function, and adding an affine-in-$s,x$ term:
\eqal{
\label{newG}
\tilde G(M,p,s,x):=G(M,p,s,x,x_0\cdot p-P(x_0))+\pd_vG|_{x_0}(s-\ell(x)),
}
where $\ell(x)=P(x_0)+(x-x_0)\cdot DP(x_0)$.  For $G$ and $\tilde G$ evaluated at $P(x)$, we have $\pd_j\tilde G=\pd_j G$ at $x_0$ since the second term is $O(x-x_0)^2$, while $v=x_0\cdot DP(x)-P(x_0)+O(x-x_0)^2$, by \eqref{jet}.  Moreover, $\pd_{ij}\tilde G=\pd_{ij}G$ at $x_0$, since we need only check the $v_{ij}$ terms:
\eqal{
\pd_{ij}(\tilde G-G)\stackrel{x_0}{=}G_v\,\pd_{ij}(x_0\cdot DP)+G_v\,\pd_{ij}\,P-G_v\pd_{ij}v,
}
which vanishes by \eqref{jet}.  
Meanwhile, since the second term in \eqref{newG} is affine, and the final argument of $G$ only uses $p$, it follows that $\tilde G$ satisfies inverse convexity \eqref{ic}.  We deduce that \eqref{super} holds using $\tilde G$.  We thus obtain
\eqal{
G^{ab}Q^k_{ab}&\stackrel{x_0}{=}\tilde G^{ab}Q^k_{ab}\stackrel{\eqref{super}}{\le} C\sum_{\ell=1}^k(Q^\ell+|DQ^\ell|+|\pd_{\ell\ell}\tilde G|)\\
&\stackrel{x_0}{=} C\sum_{\ell=1}^k(Q^\ell+|DQ^\ell|+|\pd_{\ell\ell}G|).
}
Since the second order Taylor polynomial of $\tilde G$ at $P(x_0)$ is controlled by that of $G$ at $P(x_0)$, the constant $C$ has the required uniformity.  This concludes the proof.

\end{proof}

\section{Strong convexity without constant rank theorems}
\label{sec:convex}

The following dual strong convexity criterion is sharp.  Considering the example $U(x)=|x|^{1+\beta}/(1+\beta)$ with $\beta\in(0,1)$, its Legendre transform $U^*(\bar x)=|\bar x|^{1+\beta^{-1}}/(1+\beta^{-1})$ is $C^{2,\al}$ for $\al=\beta^{-1}-1$, or $\beta=1/(1+\al)$, and fails strong convexity at the origin.  We assume $U(x)$ is strictly convex so that $\pd U(B_1)$ is open.

\begin{theorem}
\label{thm:convex}
 Let $U(x)$ be strictly convex on $B_1$ with Legendre transform $U^*(\bar x)\in C^{2+\al-}$ on $\pd U(B_1)$ for some 
 $\al\in(0,2]$.  If also $U(x)\in C^{1,\beta}$ for some $\beta>1/(1+\al)$, then $U^*$ is strongly convex, $D^2U^*>0$.
\end{theorem}

\begin{proof}
First, we will show that it suffices to prove the theorem supposing that $U^*\in C^{2+\al}$ instead of $C^{2+\al-}$, with all other hypotheses unchanged.  For if $U^*\in C^{2+\al-}$ and $U\in C^{1,\beta}$ for some $\beta>1/(1+\alpha)$, then also $\beta>1/(1+\alpha')$ for some sufficiently close $\alpha'<\alpha$.  Since $U^*\in C^{2+\al'}$, we deduce that $D^2U^*>0$, as required.

\medskip
Next, we suppose that $U^*\in C^{2+\al}$ with the other given hypotheses.  
Suppose the conclusion is false, or that if $\la^*_i$ are the eigenvalues of $D^2U^*$, then $\la^*_{min}(0)=0$ for $DU(0)=0$ and $U(0)=0$, say.  In general, we have the degenerate case
\eqal{
0=\la^*_1(0)=\cdots=\la^*_k(0)<\la_{k+1}^*(0)\le\cdots\le\la_n^*(0).
}
We may assume $D^2U^*(0)$ is diagonal, with $U^*_{ii}(0)=\la^*_i(0)$.  Let us replace $U^*(\bar x)$ with
\eqal{
\hat U(\bar x):=U^*(\bar x)+\frac{1}{2}(2\bar x_2^2+3\bar x_3^2+\cdots+n\bar x_n^2).
}
Since $D^2\hat U(\bar x)$ is $C^\al$, its spectrum $\hat\la_i(\bar x)$ is nondegenerate and $C^\alpha$ for small $\bar x$:
\eqal{
0\le\hat\la_1(\bar x)<\hat\la_2(\bar x)<\cdots<\hat\la_n(\bar x).
}
Since $\hat\la_1(0)=0$, it achieves a local minimum, so the fact it is $C^\al$ implies the decay
\eqal{
\label{upper}
\hat\la_{min}(\bar x)\le C|\bar x|^{\al},
}
even if $\al\in(1,2]$.  To find a lower bound to compare with, we use $U\in C^{1,\beta}$ to get
\eqal{
U(x)\le \frac{C_1}{1+\beta}|x|^{1+\beta}.
}
After a rescaling $C_1^{-1}U(x)$, we can assume $C_1=1$.  Let us now recall the Legendre transform is order reversing.  For $\bar x\in \pd U(B_1)$,
\eqal{
U^*(\bar x)&=\sup_{x\in B_1}(x\cdot\bar x-U(x))\ge \sup_{x\in B_1}(x\cdot\bar x-\frac{|x|^{1+\beta}}{1+\beta})\\
&=\frac{|\bar x|^{1+\beta^{-1}}}{1+\beta^{-1}},
}
the last equality holding provided $\bar x=|x|^{\beta-1}x$ is solvable for $x\in \overline B_1$, i.e $|\bar x|\le 1$.  Since $DU$ is $C^\beta$, it follows that $\pd U(B_r)\subset B_1$ for some $r$ small enough, so for $\bar x\in \pd U(B_r)$,
\eqal{
\label{lower}
\hat U(\bar x)\ge U^*(\bar x)\ge \frac{\beta}{1+\beta}|\bar x|^{1+\beta^{-1}}.
}
Let us convert \eqref{upper} into an analogous upper bound on $\hat U(x)$.  If $f(0)=f'(0)=0$, then integration by parts yields
\eqal{
\label{ibp}
f(t)&=\int_0^tf'(s)ds=\int_0^t\frac{d}{ds}(s-t)f'(s)ds\\
&=\int_0^t(t-s)f''(s)ds. 
}
Recall that $e_1=(1,0\dots,0)$ is a minimum unit eigenvector of $D^2\hat U(0)$.  Applying \eqref{ibp} to $f(t)=\hat U(t e_1)$ gives
\eqal{
\label{ibp1}
\hat U(te_1)=\int_0^t(t-s)\langle e_1,D^2\hat U(se_1)\cdot e_1\rangle \,ds.
}
From the proof of \cite[Theorem 5.1]{A}, there is a minimum unit eigenvector $e_{min}(A)$ of symmetric matrix $A$ which depends smoothly on $A$ if $spec(A)$ is nondegenerate.  Denoting $e_{min}(\bar x)$ the one for $A=D^2\hat U(\bar x)$ which satisfies $e_{min}(0)=e_1$, we obtain
\eqal{
\|e_{min}(\bar x)-e_1\|&\le C\|D^2\hat U(\bar x)-D^2\hat U(0)\|\\
&\le C|\bar x|^{\min(\al,1)}.
}
The same will be true for the projection $e(\bar x)$ of $e_1$ onto the $e_{min}(\bar x)$ subspace.  That is, denoting $\delta=e(\bar x)-e_1$, we have
$$
\|\delta\|\le C|\bar x|^{\min(\alpha',1)},\qquad \langle\delta,e\rangle=0,\qquad |e|^2=1-|\delta|^2,\qquad D^2\hat U(\bar x)\cdot e(\bar x)=\hat \lambda_{min}(\bar x)e(\bar x).
$$
For $\bar x=s\,e_1$, we substitute these relations into \eqref{ibp1} and obtain
\eqal{
\hat U(te_1)&=\int_0^t(t-s)(\hat\la_{min}(se_1)|e|^2+\langle \delta,D^2\hat U(se_1)\cdot\delta\rangle)\,ds\\
&\le\int_0^t(t-s)\hat\la_{min}(se_1)ds+Ct^{2\min(\al,1)}\int_0^t(t-s)ds\\
&=\int_0^t(t-s)\hat\la_{min}(se_1)ds+Ct^{2+2\min(\al,1)}.
}
Recalling \eqref{upper}, we thus obtain the upper bound
\eqal{
\hat U(te_1)\le Ct^{2+\al}.
}
Combining this with the lower bound \eqref{lower}, we obtain, for small $t$,
\eqal{
Ct^{1+1/\beta}\le Ct^{2+\al}.
}
Since $\beta>1/(1+\al)$ or $2+\al>1+1/\beta$, it follows that we can choose $t$ small and obtain a contradiction.
\end{proof}

\section{Proof of Theorems \ref{thm1} and \ref{thm2}}
\label{sec:proofs}

In both cases, it suffices to establish that $\bar\la_{max}(\bar x)<1$ on $\pd \tilde u(B_1)$.  By comparing with quadratics as in \cite{CSY} and \cite{BS1} and using the order preservation of rotation, we deduce that $\la_{max}(x)<\infty$ on $B_1$, i.e. $u(x)\in C^{1,1}$.  We can then either invoke Evans-Krylov-Safonov theory for now-uniformly elliptic \eqref{slag} to deduce $u\in C^{2+\al-}$, or the approach in \cite{BS1} using the transformation rule
$$
D^2u(x)=\left(\frac{1+t}{1-t}\right)\circ D^2\bar u\circ \bar x(x),
$$
where $\bar x(x)=cx+sDu(x)$, and the regularity for $D^2\bar u(\bar x)$.

\medskip
Defining auxiliary convex functions
\eqal{
\label{aux}
U(x)=su(x)+\frac{c}{2}|x|^2,\qquad \bar U(\bar x)=-s\bar u(\bar x)+\frac{c}{2}|\bar x|^2,
}
where again $c=\cos\pi/4$ and $s=\sin\pi/4$, it follows from rotation formula \eqref{rot} that $\bar U(\bar x)=U^*(\bar x)$.  Since $\bar\la_{max}<1$ if and only if $D^2\bar U>0$, we will show $\bar U(\bar x)$ is strongly convex.  

\medskip
\textbf{Proof of Theorem \ref{thm2}}: If $\psi\in C^\al$ for some 
$\al\in(0,2]$ and $u(x)\in C^{1,\beta}$ for some $\beta>1/(1+\al)$, then rotated regularity Theorem \ref{thm:reg_rot} shows that $\bar u(\bar x)\in C^{2+\al-}$.  Moreover, $U(x)$ in \eqref{aux} is strictly convex, so the fact that the graph $(x,Du(x))$ has large enough exponent $\beta>1/(1+\al)$ allows us to apply strong convexity Theorem \ref{thm:convex} to deduce that $D^2U^*>0$.

\medskip
\textbf{Proof of Theorem \ref{thm1}}: We define a $C^2$ concave replacement of arctangent \eqref{slag} by
\eqal{
A(\la)=\begin{cases}\arctan\la,&\la\ge 0,\\\la,&\la<0\end{cases}.
}
To rewrite \eqref{slag} in terms of auxiliary $U(x)$ in \eqref{aux}, this changes all $u,Du$ arguments by a rescaling and a function of $x$.  The partial convexity condition \eqref{convex} is unchanged by rescaling its arguments or perturbing its arguments by $x$, so by changing $\psi$ to $\psi_1$, we rewrite \eqref{slag} as
\eqal{
\sum_{i=1}^nA(s^{-1}\La_i-1)=\psi_1(x,U,DU),
}
where $\La_i$'s are the eigenvalues of $D^2U(x)$.  Here, $\psi_1$ also satisfies partial convexity \eqref{convex} in its last argument.  Applying the Legendre transform, this equation can be written as
\eqal{
\sum_{i=1}^n-A(\frac{1}{s\bar\La_i}-1)+\psi_1(D\bar U,\bar x\cdot D\bar U-\bar U,\bar x)=0,
}
where $\bar\La_i$'s are the eigenvalues of $D^2\bar U(\bar x)$.  More precisely, $\bar U(\bar x)$ is a viscosity solution of this equation, which is the content of subsolution preservation [Proposition \ref{prop:sub}], and by rotated regularity Theorem \ref{thm:reg_rot}, we know $\bar U(\bar x)\in C^{3}$.  We will apply the constant rank theorem \ref{thm:rank} to this equation, where
\eqal{
G(M,p,x,v):=\sum_{i=1}^n-A(\frac{1}{s\la_i(M)}-1)+\psi_1(p,v,x).
}
This operator is elliptic and inverse convex since $t\mapsto-A(t)$ is convex, see e.g. \cite[Corollary 5.5]{A}.  To verify $G$ is $C^2$ at nonnegative $M$ with zero eigenvalues, we note that as $t\to 0$,
\eqal{
A(\frac{1}{st}-1)=\arctan(\frac{1}{st}-1)=\pi/2-\frac{st}{1-st}+O(t^3),
}
so $A$ is a smooth function on $[0,\infty)$, and it is known, from e.g. \cite[Theorem 5.1]{A}, the smoothness of $G$ follows from the permutation symmetry in the $\la_i$'s.  Thus, by the constant rank theorem, either $D^2\bar U(\bar x)>0$ everywhere, or $\bar\La_{min}=0$ everywhere, in particular.  The latter case is impossible: since $U(x)$ is bounded in $B_1$, we can find a tall quadratic $Q$ touching $U$ from above at some point $p\in B_1$.  
In fact, the Legendre transform is order reversing: if $\bar x\in \pd Q(B_1)\cap \pd U(B_1)$, then
\eqal{
Q^*(\bar x)=\sup_{x\in B_1}(x\cdot\bar x-Q(x))\le \sup_{x\in B_1}(x\cdot\bar x-U(x))=U^*(\bar x).
}
Since contact at $p$ shows that $DQ(p)\in \pd Q(B_1)\cap \pd U(B_1)$, the 
strong convexity of $Q$ and $U$ ensures $\pd Q(B_1)\cap \pd U(B_1)$ contains a ball centered at $\bar p:=DQ(p)$, so $Q^*$ touches $U^*$ at $\bar p$ from below near $\bar p$.  Since $D^2Q<\infty$, we have $0<D^2Q^*\le D^2U^*(\bar p)$, a contradiction.

\section{Singular solution to self similar equation \eqref{self}}
\label{sec:self}

In dimension $n=3$, Wang-Yuan \cite{WangY} constructed a $C^{1,1/3}$ solution $u(x)$ on $B_1$ to special Lagrangian equation \eqref{s1} which is smooth away from the origin $x=0$, originally found in Nadirashvili-Vl\u{a}du\c{t} \cite{NV}.  The idea was to find a polynomial approximate solution $P(\bar x)$ in rotated $(\bar x,\bar y)$ coordinates, invoke the Cauchy-Kovalevskaya Theorem to obtain a full solution, then apply a $U(n)$ transformation to return to $(x,y)$ coordinates.

\smallskip
In this section, we show that this solution, essentially, solves the self-shrinker/expander equation \eqref{self}.  More precisely, in rotated coordinates, we will see that $P(\bar x)$ above approximately solves the associated self-similar equation in rotated coordinates.  Therefore, the self-similar corrections only occur at higher order.

\begin{proposition}
There exists a $C^{1,1/3}(B_r)\cap C^\infty(B_r\setminus\{0\})$ singular viscosity solution to the self-similar Lagrangian mean curvature flow potential equation \eqref{self} for some small $r>0$ and $B_r\subset\re^3$.
\end{proposition}

\begin{proof}
Given $\Theta\in[0,\pi/2)$, by the construction in \cite[pg 1165]{WaY}, there is an analytic solution $\bar u(\bar x)$ to special Lagrangian equation \eqref{s1} with phase $c=\pi/2-\Theta$ near $\bar x=0$.  The canonical angles $\bar\theta_i=\arctan\lambda_i(D^2\bar u)$ satisfy near $x=0$, by \cite[Property 3.3, pg 1167]{WaY},
\eqal{
\label{angles}
\bar\theta_1,\bar\theta_2&=(\frac{\pi}{4}-\frac{\Theta}{2})(1+O(\bar x)),\\
\bar\theta_3&=-\frac{1}{\tan(\frac{\pi}{4}-\frac{\Theta}{2})}\,Q(\bar x)(1+O(\bar x)),
}
where $Q(\bar x)=\frac{1}{2}\langle \bar x,D^2Q\cdot\bar x\rangle$ is positive definite.  Note that we may assume $\bar u(0)=0$ and $D\bar u(0)=0$, we can write $\bar u(\bar x)=\bar u_2(\bar x)+O(|\bar x|^3)$, where $\bar u_2$ is a homogeneous quadratic polynomial.  Substituting this into the self-similar operator \eqref{self} yields
\eqal{
\theta(D^2\bar u)-c-b(\bar x\cdot D\bar u-2\bar u)&=-b(\bar x\cdot D\bar u-2\bar u)\\
&=O(|\bar x|^3).
} 
It follows from the Cauchy-Kovalevskaya Theorem that there exists an analytic solution $\bar v(\bar x)$ to the self-similar equation with $c=\pi/2-\Theta$ which agrees with $\bar u(\bar x)$ up to and including $O(|\bar x|^4)$ terms.  Since only these terms contribute to the $O(|\bar x|^2)$ angle asymptotics \eqref{angles}, this does not alter the Wang-Yuan construction, and we can continue their procedure unchanged to obtain a singular viscosity solution $v(x)$.  

\smallskip
It only needs to be verified that the equation $v(x)$ solves is still a self-similar equation since in \cite{WaY}, the special Lagrangian equation is considered. This follows from the representation $\bar X(\bar x,t)=\sqrt{1-2bt}\,(\bar x,D\bar v(\bar x))$ of the self-similar solution we obtain in rotated coordinates.  Under a $U(n)$ transformation of space, the solution still has this separated time dependence, $X(x,t)=U\cdot X(\bar x,t)=\sqrt{1-2bt}\,(x,Dv(x))$, so the spatial potential $v(x)$ still solves a self-similar potential equation, with different constant $c=\Theta$.  Note that the solution is only defined on $B_r$ since the Wang-Yuan solution on $B_1$ arises from one on $B_r$ after a quadratic rescaling $r^2u(\frac{x}{r})$, which is not possible in the self-similar setting.
\end{proof}

\section{Singular solution to rotator equation \eqref{rotator}}
\label{sec:rot}

In dimension $n=1$, Hungerb\"{u}hler-Smoczyk \cite{HS} constructed a global ``Yin-Yang" curve rotating soliton solution of mean curvature flow, $H=(JX)^\bot$.  Locally, its Lagrangian potential will satisfy the rotator potential equation \eqref{rotator}.  For all dimensions, we provide a local construction of this solution near the center of the spiral which manifests a singularity in the potential, for a poor choice of coordinates.

\begin{proposition}
There exists a $C^{1,1/3}(B_r)\cap C^\infty(B_r\setminus\{0\})$ singular viscosity solution to the rotator Lagrangian mean curvature flow potential equation \eqref{rotator} if $a<0$ and $c=n\pi/2$, for some small $r>0$.
\end{proposition}

\begin{proof}
We first construct an approximately radial analytic solution of the following ``rotated" rotator equation with smaller phase constant $c$:
\eqal{
\label{rotrot}
\sum_{i=1}^n\arctan\bar\la_i=n\pi/4+\frac{a}{2}(|\bar x|^2+|D\bar u(\bar x)|^2).
}
A radial solution $\bar u(\bar x)=f(\rho)$, $\rho=|\bar x|^2/2$, will solve the equation
\eqal{
(n-1)\arctan(f'(\rho))+\arctan(2\rho f''(\rho)+f'(\rho))=n\pi/4+a\rho(1+f'(\rho)^2).
}
We choose $f(0)=0$, so setting $\rho=0$ yields $f'(0)=1$.  Differentiating this equation at $\rho=0$ yields
\eqal{
(n-1)\frac{f''(0)}{1+f'(0)^2}+\frac{3f''(0)}{1+f'(0)^2}=a(1+f'(0)^2).
}
Thus $f''(0)=4a/(n+2)$.  If we denote rotator equation \eqref{rotrot} by $F(\bar x,D\bar u,D^2\bar u)=0$ and our approximate solution by
\eqal{
\bar v(\bar x):=\frac{1}{2}|\bar x|^2+\frac{a}{2(n+2)}|\bar x|^4,
}
then this calculation shows that $F|_{\bar v}=O(\rho^2)$, hence $\bar v(\bar x)$ solves \eqref{rotrot} to $O(|\bar x|^4)$.  It follows from the Cauchy-Kovalevskaya Theorem that there exists an analytic solution $\bar u(\bar x)$ to \eqref{rotrot} on a small ball $B_{\bar r}$ which agrees with $\bar v(\bar x)$ until $O(|\bar x|^6)$ terms.  

\smallskip
We now perform the inverse rotation operation $z=e^{i\pi/4}\bar z$.  Let us first verify that a local potential exists.  The auxiliary function $\bar U(\bar x)$ in \eqref{aux} has the expansion
\eqal{
\bar U(\bar x)=-\frac{as}{2(n+2)}|\bar x|^4+O(|\bar x|^6).
}
Since $a<0$, we see $\bar U$ is strictly convex and has an inverse Legendre transform $U(x)$ on some small ball $B_r$, with a fractional power leading term:
\eqal{
\label{lead}
U(x)=C|x|^{4/3}+O(|x|^2).
}  
Moreover, since $\bar U(\bar x)$ is strongly convex away from the origin, it follows $U(x)$ is smooth away from the origin.  Defining $u(x)$ in terms of $U(x)$, as in \eqref{aux}, we recall the $(x,Du)$ and Hessian transformation rules \eqref{rotx} and \eqref{angle}.  The Euclidean length $|x|^2+|y|^2$ is preserved under rotations, and the angles $\theta_i=\arctan\la_i$ increase under rotations, so $u(x)$ solves the rotator equation
\eqal{
\sum_{i=1}^n\arctan\la_i=n\pi/2+\frac{a}{2}(|x|^2+|Du(x)|^2),
}
in the classical sense away from the origin.  Moreover, by the expansion \eqref{lead}, it follows that it is also a viscosity solution at the origin.
\end{proof}

Let us note another way to see the ``rotator" property of the potential is preserved.  The rotator potential $\bar u(\bar x)$ leads to a mean curvature flow by $J$, $\bar X(\bar x,t)=e^{atJ}(\bar x,D\bar u(\bar x))$, which commutes with $J$-rotations: $X(x,t):=e^{\frac{\pi}{4}J}\bar X(\bar x,t)=e^{atJ}(x,Du(x))$, where $X(x,0)=e^{\frac{\pi}{4}J}\bar X(\bar x,0)$ is the rotation of the initial submanifold.  This means $X(x,t)$ is also a rotator solution of Lagrangian mean curvature flow, so the initial potential $u(x)$ of $X(x,t)$ solves the rotator potential equation, with a phase larger by $n\pi/4$.

\bibliographystyle{amsalpha}
\bibliography{BiB}

\end{document}